\documentclass[runningheads]{llncs}
\usepackage[T1]{fontenc}
\usepackage{graphicx}
\usepackage{amssymb,amsmath}

\newcommand{\om}{\Omega}

\newcommand{\R}{\mathbb{R}}

\newcommand{\be}{\begin{eqnarray}}
\newcommand{\ee}{\end{eqnarray}}
\newcommand{\av}{-\hspace{-.15in}\int}
\newcommand{\avsmall}{-\hspace{-.112in}\int}

\renewcommand{\leq}{\leqslant}
\renewcommand{\geq}{\geqslant}

\newcommand{\cof}{{\rm cof}}

\newcommand{\1}{{\mathbf 1}}

\begin{document}
\title{Image comparison and scaling via nonlinear elasticity}
%
%
\author{John M. Ball
 \and
Christopher L. Horner
%
\authorrunning{J.M. Ball et al.}
%
\institute{ Heriot-Watt University and Maxwell Institute for the Mathematical Sciences, Edinburgh, U.K.}
\\
}

\maketitle              
\begin{abstract}
A nonlinear elasticity model for comparing images is formulated and analyzed, in which optimal transformations between images are sought as minimizers of an integral functional. The existence of minimizers in a suitable class of homeomorphisms between image domains is established under natural hypotheses. We investigate whether for linearly related images  the minimization algorithm delivers the linear transformation as the unique minimizer.

\keywords{Nonlinear elasticity  \and image registration \and scaling.}
\end{abstract}
\section{Introduction}
In this paper we formulate and analyze a nonlinear elasticity model for comparing two images $P_1=(\Omega_1,c_1),\; P_2=(\Omega_2,c_2)$, regarded as bounded Lipschitz domains $\om_1,\om_2$ in $\R^n$ with corresponding intensity maps $c_1:\om_1\to\R^m, c_2:\om_2\to\R^m$. The model is based on an integral functional
\be
\label{0} I_{P_1,P_2}(y)=\int_{\Omega_1}\psi(c_1(x), c_2(y(x)),Dy(x)) \,dx,
\ee
 depending on $c_1,c_2$ and  a map $y:\om_1\to\om_2$ with  gradient $Dy$, whose minimizers give  optimal transformations $y^*$ between images. The admissible transformations $y$ between the images are orientation-preserving homeomorphisms with $y(\om_1)=\om_2$, and are not required to satisfy other boundary conditions. 

The use of nonlinear elasticity, rather than models based on linear elasticity that are more commonly used in the computer vision literature, provides a conceptually clearer and more general framework. A key advantage is that nonlinear elasticity (of which linear elasticity is not a special case) respects rotational invariance, so that rigidly rotated and translated images can be identified as equivalent. Further, nonlinear elasticity is naturally suited for discussing the global invertibility of maps between images (see, for example, \cite{j16}, \cite{Sverak88}), which in the context of mechanics describes non-interpenetration of matter.  

Our work is closest in spirit to that of Droske \& Rumpf \cite{droskerumpf04}, Rumpf \cite{rumpf2013} and Rumpf \& Wirth \cite{rumpfwirth}, who like us make use of the existence theory for polyconvex energies in \cite{j8}, as also do Burger, Modersitski \& Ruthotto \cite{burgeretal2013}, Debroux et al \cite{debrouxetal20}, Iglesias, Rumpf \& Scherzer \cite{iglesiasrumpfscherzer} and Iglesias \cite{iglesias21}. Other nonlinear elasticity approaches are due to  Lin, Dinov, Toga \& Vese \cite{linetal2010},  Ozer\'{e}, Gout \& Le Guyader \cite{OzereLeguyader2015},  Ozer\'{e} \& Le Guyader \cite{OzereLeguyader2015a}, Simon, Sheorey, Jacobs \& Basri \cite{simonetal2017} and  Debroux \& Le Guyader\cite{debroux}.  Key differences with these works are:\\ (i) that we minimize among homeomorphisms of the image domains rather than applying Dirichlet or other boundary conditions,\\ (ii) technical improvements as regards the regularity of the intensity maps, and\\ (iii) a novel analysis of linearly related images.

Our model is described in Section \ref{nle}, in which it is shown (Proposition \ref{prop1}) that invariance of the integral \eqref{0} under rotation and translation requires that the integrand $\psi(c_1,c_2,\cdot)$ be isotropic.  As described above, two images that are translated and rigidly rotated with respect to each other can reasonably be regarded as equivalent. In most applications the minimization algorithm should thus deliver this translation and rotation as the unique minimizer, and we give conditions on $\psi$ under which this occurs.  We also discuss symmetry with respect to interchange of images. Theorem \ref{exthm} gives the existence of  a minimizer  for general pairs of images under polyconvexity and growth conditions on $\psi$, assuming only that the intensity maps are $L^\infty$.

More generally we consider the case when two images are related by a linear transformation, and ask for which $\psi$ the minimization algorithm delivers this linear transformation as the unique minimizer. We show  (see Section \ref{scaling}) that  $\psi$  can be chosen such that for any pair of images related by a uniform magnification the unique minimizer is that magnification. However, for the functional to deliver as a minimizer the linear transformation between {\it any} linearly related pair of images the integrand must have a special form (see Theorem \ref{generalM}), in which   the integrand depends on the gradient $Dy$  as a convex function of $\det Dy$ alone. This degeneracy suggests that a better model might use an integrand depending also on the second gradient $D^2y$, and this is briefly discussed, together with other issues, in Section \ref{discussion}.

\section{Nonlinear elasticity model}
\label{nle}
\subsection{Comparing images}
We identify an image with a pair $P=(\Omega,c)$, where $\Omega\subset\R^n$ is a bounded Lipschitz domain, and $c:\Omega\to\R^m$ is an {\it intensity map} describing the greyscale intensity ($m=1$),  the intensity of colour channels ($m>1$), and possibly other image characteristics. Our aim is to compare two images 
$P_1=(\Omega_1,c_1),\; P_2=(\Omega_2,c_2)$ 
by means of a nonlinear elasticity based functional, whose minimizers give  optimal transformation maps between the images.

To compare $P_1, P_2$ we minimize the functional
\be
\label{1}I_{P_1,P_2}(y)=\int_{\Omega_1}\psi(c_1(x), c_2(y(x)),Dy(x)) \,dx,
\ee
over invertible maps $y:\Omega_1\to\R^n$ such that $y(\Omega_1)=\Omega_2$, and which are {\it orientation-preserving},  that is $\det Dy(x)>0$ for a.e. $x\in\om_1$. 
Here 
$$\psi:\R^m\times\R^m\times M^{n\times n}_+\to [0,\infty),$$
where
$M^{n\times n}_+=GL^+(n,\R)=\{{\rm real}\; n\times n \,{\rm matrices}\; A\;{\rm with}\;\det A>0\}$.

In \eqref{1}, $Dy(x)$ denotes the distributional gradient of $y$ at $x$. Throughout this section we assume that the maps $y$ and their inverses $y^{-1}$ have sufficient regularity; it is enough that $y\in W^{1,p}(\Omega_1, \R^n), y^{-1}\in W^{1,p}(\Omega_2,\R^n)$ for $p>n$ (for the definition of $W^{1,p}$ see Section \ref{existence}), which is guaranteed by Theorem \ref{exthm} below.

Note that we do not specify $y$ on $\partial \Omega_1$, only that $y(\Omega_1)=\Omega_2$. Thus we allow `sliding at the boundary', in order to better compare images with important boundary features. This is not typically done in the computer vision literature, but is considered in the context of elasticity by Iwaniec \& Onninen \cite{iwanieconninen2009}, though for elasticity such a boundary condition would be difficult to realize mechanically.

\subsection{Properties of $\psi$}
\label{psiprops}
We now list some desirable properties of the integrand $\psi$ in \eqref{1}.\vspace{.05in}\\ 
(i){ \it Invariance under rotation and translation.}
For two images $P=(\Omega,c)$ and $P'=(\Omega',c')$ write $P\sim P'$ if $P,P'$ are related by a rigid translation and rotation, i.e.
$$\Omega'=E(\Omega),\; c'(E(x))=c(x)$$
for some proper rigid   transformation $E(x)=a+Rx$, $a\in\R^n$, $R\in SO(n)$.

If $P_1\sim P_1'$, $P_2\sim P_2'$, with corresponding rigid transformations $E_1(x)=a_1+R_1x,\, E_2(x)=a_2 +R_2x$,
we require that 
\be
\label{1a} 
I_{P_1,P_2}(y)=I_{P_1',P_2'}(E_2\circ y\circ E_1^{-1}),
\ee
or, equivalently,
\be
\label{2}
\int_{\Omega_1}\psi(c_1(x),c_2(y(x)), R_2Dy(x)R_1^T)\,dx&\nonumber\\
&\displaystyle\hspace{-.5in}=\int_{\Omega_1}\psi(c_1(x),c_2(y(x)),Dy(x))\,dx.
\ee
\begin{proposition}
\label{prop1}
\eqref{2} holds for all $P_1,P_2$ and orientation-preserving invertible $y:\Omega_1\to\Omega_2$ with $y(\Omega_1)=\Omega_2$ iff $\psi(c_1,c_2,\cdot)$ is {\it isotropic}, i.e.
\be
\label{3}
 \psi(c_1,c_2,QAR)=\psi(c_1,c_2,A)
 \ee
   for all  $c_1,c_2\in \R^n, A\in M^{n\times n}_+$, and $Q,R \in SO(n)$.
\end{proposition}
\begin{proof}
Setting $c_1, c_2$ constant, and $y(x)=Ax$, \eqref{2} implies \eqref{3}, and the converse is obvious.
\end{proof}
We denote by $v_i(A)$  the singular values of $A$ (that is, the eigenvalues of $\sqrt {A^TA}$). A standard result of nonlinear elasticity (see, for example, \cite[Theorem 8.5.1]{silhavy97a}) gives that $\psi(c_1,c_2,\cdot)$ is isotropic iff
$$\psi(c_1,c_2, A)=H(c_1,c_2,v_1(A),...,v_n(A))$$
with $H$ symmetric with respect to permutations of the last $n$ arguments.\vspace{.05in}

\noindent (ii) {\it Matching of equivalent images}. We also require that the functional \eqref{1} is zero iff the two images are related by a rigid transformation, i.e. for invertible $y$ with $y(\Omega_1)=\Omega_2$ we have
\be
\label{4}
 I_{P_1,P_2}(y)=0  \text{ iff } P_1\sim P_2 \text{ with corresponding rigid transformation } y.
 \ee
\begin{proposition}
\eqref{4} is equivalent to the condition
\be
\label{5}
\psi(c_1,c_2,A)=0 \text{ iff }c_1=c_2\text{ and }A\in SO(n).
\ee
\end{proposition}
\begin{proof}
The  only nontrivial part of the proof of equivalence is to show that if \eqref{5} holds and $I_{P_1,P_2}(y)=0$ then $P_1\sim P_2$ with corresponding rigid transformation $y$. But if $I_{P_1,P_2}(y)=0$ then $c_2(y(x))=c_1(x)$ and $Dy(x)=R(x)\in SO(n)$ for a.e. $x\in\Omega_1$. But this implies by \cite{reshetnyak67} that $R(x)=R$ is constant, from which the conclusion follows.
\end{proof}

\noindent(iii) {\it Symmetry with respect to interchanging images.}
For applications in which both images are of the same type (but not, for example, when $P_1$ is a template image) it is reasonable to require that 
\be
\label{6}
I_{P_1,P_2}(y)=I_{P_2,P_1}(y^{-1}).
\ee
Equivalently
\be
\label{7}
\int_{\Omega_1}\psi(c_1(x),c_2(y(x)),Dy(x))\,dx
&=&\int_{\Omega_2}\psi(c_2(y),c_1(x(y)),Dx(y))\,dy\\
&=&\int_{\Omega_1}\psi(c_2(y(x)),c_1(x), Dy(x)^{-1})\det Dy(x)\,dx.\nonumber
\ee
Taking $c_1,c_2$ constant and $y(x)=Ax$ this holds iff
\be
\label{8}\psi(c_1,c_2,A)=\psi(c_2,c_1,A^{-1})\det A.
\ee
Such a symmetry condition was introduced by Cachier \& Rey \cite{cachierrey2000} and subsequently used by Kolouri, Slep\v{c}ev \& Rohde \cite{kolourietal2015} and Iglesias \cite{iglesias21}. It is also implicit in the work of Iwaniec \& Onninen \cite{iwanieconninen2009}.

A class of integrands $\psi$ satisfying the above conditions \eqref{3}, \eqref{5}, \eqref{8} is given by
\be
\label{8aa}
\psi(c_1,c_2,A)=\Psi(A)+f(c_1,c_2,\det A),
\ee
where\\ (a) $\Psi\geq 0$ is isotropic, $\Psi(A)=\det A \cdot\Psi(A^{-1})$, $\Psi^{-1}(0)=SO(n)$, \\(b) $f\geq 0$, $f(c_1,c_2,\delta)=\delta f(c_2,c_1,\delta^{-1})$, $f(c_1,c_2,1)=0$ iff $c_1=c_2$.

In particular we can take
\be 
\label{8a}f(c_1,c_2,\delta)=(1+\delta)|c_1-c_2|^2,
\ee
or
\be
\label{8b}
f(c_1,c_2,\delta)= |c_1-c_2\delta|^2+\delta^{-1}|c_1\delta-c_2|^2,
\ee
which are both convex in $\delta$.
\subsection{Existence of minimizers}
\label{existence}
  Let $p>n$ and define the set of admissible maps
\be
{\mathcal A}=\{y\in W^{1,p}(\Omega_1,\R^n): y:\Omega_1\to \Omega_2 \text{ an orientation-preserving}&\\ &\hspace{-1.5in}\text{homeomorphism},  y^{-1}\in W^{1,p}(\Omega_2,\R^n)\}.\nonumber
\ee
Here, for a bounded domain $\om\subset\R^n$ and $1<p<\infty$, $W^{1,p}(\Omega,\R^n)$ is the Sobolev space of maps $y:\om\to\R^n$ such that
$$\|y\|_{1,p}:=\left(\int_\om\left(|y(x)|^p+|Dy(x)|^p\right)\,dx\right)^\frac{1}{p}<\infty.$$
 We recall (see for example \cite{adamsfournier03,mazya2011}) that if $\om$ is Lipschitz and $p>n$ then any $y\in W^{1,p}(\Omega,\R^n)$ has a representative that is continuous on the closure $\bar\om$ of $\om$.

We now make some other technical hypotheses on $\psi$.\vspace{.05in}\\
(H1) ({\it Continuity})  $\psi:\R^m\times \R^m\times M^{n\times n}_+\to [0,\infty)$ is continuous,\vspace{.05in}\\
(H2) ({\it Coercivity})  $\psi(c,d,A)\geq C(|A|^p+\det A\cdot|A^{-1}|^p) -C_0$\\ for all $c,d\in\R^m, A\in M^{n\times n}_+$, where $C>0$ and $C_0$ are constants,\vspace{.05in}\\
(H3) ({\it Polyconvexity}) $\psi(c,d,\cdot)$ is polyconvex for each $c,d\in\R^s$, i.e. there is a function $g:\R^m\times\R^m\times \R^{\sigma(n)}\times (0,\infty)\to \R$ with $g(c,d,\cdot)$ convex, such that  
$$\psi(c,d,A)=g(c,d,{\mathbf J}_{n-1}(A), \det A)\text{ for all }c,d\in\R^m, A\in M^{n\times n}_+,$$ where ${\mathbf J}_{n-1}(A)$ is the list of all minors (i.e. subdeterminants) of $A$ of order $\leq n-1$ and $\sigma(n)$ is the number of such minors,\vspace{.05in}\\
(H4) ({\it Bounded intensities})  $c_1\in L^\infty(\Omega_1,\R^s)$, $c_2\in L^\infty(\Omega_2,\R^s)$.\vspace{.05in}

We note that (H2) implies that $\psi(c,d,A)\geq C_1 (\det A)^{1-\frac{p}{n}}$ for some constant $C_1>0$, so that $\psi(c,d,A)\to\infty$ as $\det A\to 0+$. This follows from the Hadamard inequality $|B|^n\geq n^\frac{n}{2}\det B$ for $B\in M^{n\times n}_+$ applied to $B=\cof A$, noting that $\det \cof A=(\det A)^{n-1}$.

\begin{theorem}
\label{exthm}Suppose that $\mathcal A$ is nonempty, and that the hypotheses {\rm (H1)-(H4)} hold. Then there exists an absolute minimizer $y^*$ in $\mathcal A$ of
$$I_{P_1,P_2}(y)=\int_{\Omega_1}\psi(c_1(x), c_2(y(x)),Dy(x)) \,dx.$$
\end{theorem}
The proof, which will appear in \cite{ballhornermaths},  follows the  usual pattern for proving existence of minimizers in nonlinear elasticity for a polyconvex stored-energy function using the direct method of the calculus of variations (see \cite{j8,ciarlet83,silhavy97a}). However there are some extra issues. In particular, as observed by Rumpf \cite{rumpf2013} care has to be taken for intensity maps $c_1,c_2$ that are discontinuous, for which it is not even immediately obvious that  $I_{P_1,P_2}(y)$ is well defined, and we are able to weaken his hypotheses by assuming only (H4). Note that the hypotheses on $\psi$ discussed in Section \ref{psiprops} are not needed to prove the existence of minimizers.

\section{Linear scaling}
\label{scaling}
Suppose that the images $P_1=(\Omega_1,c_1)$ and $P_2=(\Omega_2,c_2)$ are linearly related, i.e. for some   $M\in M^{n\times n}_+$ we have
\be
\label{lin}\Omega_2=M\Omega_1,\;\; c_2(Mx)=c_1(x).
\ee
Can we choose $\psi$ such that the unique minimizer $y$ of $I_{P_1,P_2}$ is $y(x)=Mx$?

For simplicity consider  $\psi$ of the form \eqref{8aa}, \eqref{8a}
$$\psi(c_1,c_2,A)=\Psi(A)+(1+\det A)|c_1-c_2|^2.$$
Thus we require that  for all orientation-preserving invertible  $y$ with $y(\Omega_1)=M\Omega_1$
\be
\int_{\Omega_1}\left(\Psi(Dy(x))+(1+\det Dy(x))|c_1(x)-c_2(y(x))|^2\right)\,dx \geq \displaystyle\int_{\Omega_1}\Psi(M)\,dx,\;\;
\ee
with equality iff $y(x)=Mx$.

This holds for all $c_1, c_2$ iff
\be
\label{qc}\av_{\Omega_1}\Psi(Dy)\,dx\geq \Psi(M)
\ee
for $y$ invertible with $y(\Omega_1)=M\Omega_1$, where $\avsmall_{\Omega_1} f\,dx:=\frac{1}{|\om_1|}\int_{\om_1}f\,dx$ and  $|\om_1|$ is the $n-$dimensional Lebesgue measure of $\om_1$. The inequality \eqref{qc} is a stronger version of {\it quasiconvexity at }$M$, the central convexity condition of the multi-dimensional calculus of variations  implied by polyconvexity (see, e.g. \cite{rindler2018}), in which the usual requirement that $y(x)=Mx$ for $x\in\partial\Omega_1$ is weakened.

We show that we can satisfy this condition if $M=\lambda \1$, $\lambda>0$, where $\1$ denotes the identity matrix (or more generally if $M=\lambda R$, $R\in SO(n)$), so that $P_2$ is a uniform magnification (or reduction if $\lambda\leq 1$) of $P_1$. Let 
\be\Psi(A)= \sum_{i=1}^nv_i^\alpha  +
(\det A)\sum_{i=1}^nv_i^{-\alpha}+h(\det A),
\ee
where $v_i=v_i(A)$ are the singular values of $A$, $\alpha>n$, and where $h:(0,\infty)\to\R$ is $C^1$, convex and bounded below with $h(\delta)=\delta h(\delta^{-1})$ and $h'(1)=-n$.
Then $\Psi$ is isotropic, $\Psi(A)=\det A\cdot \Psi(A^{-1})$,
$\Psi\geq 0$, $\Psi^{-1}(0)=SO(n)$, and $\psi$ satisfies (H1)-(H4).

Let $y$ be invertible with $y(\Omega_1)=\lambda\Omega_1$. By the arithmetic mean -- geometric mean inequality we have that, since $\det Dy=\prod_{i=1}^nv_i$,
\begin{eqnarray*}
\av_{\om_1}\Psi(Dy)\,dx&\geq& \av_{\Omega_1}n\left((\det Dy)^{\frac{\alpha}{n}}+(\det Dy)^{1-\frac{\alpha}{n}}\right)+h(\det Dy)\,dx\\
&=& \av_{\Omega_1}H(\det Dy(x))\,dx\\
&\geq & H\left(\av_{\Omega_1} \det Dy(x)\,dx\right)\\
&=&H(\lambda^n)= \Psi(\lambda{\bf 1}),
\end{eqnarray*}
as required, where we have set $H(\delta):=n(\delta^\frac{\alpha}{n}+\delta^{1-\frac{\alpha}{n}})+h(\delta)$ and used Jensen's inequality, noting that $H$ is convex and that $\int_{\om_1}\det Dy(x)\,dx$ is the $n$-dimensional measure of $y(\om_1)$.

We have equality only when each $v_i=\lambda$, i.e. $Dy(x)=\lambda R(x)$ for $R(x)\in SO(n)$, which implies that $R(x)=R$ is constant and $a+\lambda R\Omega_1=\lambda\Omega_1$, for some $a\in\R^n$, which for generic $\Omega_1$ implies $a=0$ and $R=\1$, hence $y(x)=\lambda x$.

However, for \eqref{qc} to hold for general $M$ implies that $\Psi$ has a special form:
\begin{theorem}[{\rm see \cite{ballhornermaths}}]
\label{generalM}   $$\av_{\Omega_1}\Psi(Dy)\,dx\geq \Psi(M)$$
for all orientation-preserving invertible $y$  with $y(\Omega_1)=M\Omega_1$, and   for every $\Omega_1$ and $M\in M^{n\times n}_+$,  iff
$$\Psi(A)=H(\det A)$$
for some convex $H$.
\end{theorem}
{\it Sketch of proof.}
If $y=Mx$ is a minimizer, then we can construct a variation that slides at the boundary, so that the tangential component at the boundary of the `Cauchy stress' is zero,
i.e. 
$$D\Psi(M)M^T=p(M)\1,$$
for a scalar $p(M)$,
from which it follows that $\Psi$ corresponds to an elastic fluid, i.e. $\Psi(M)=H(\det M)$. But then $H(\det M)$ is quasiconvex, and so $H$ is convex.

Conversely, if $H$ is convex then 
\be \av_{\Omega_1}H(\det Dy(x))\,dx&\geq &H\left(\av_{\Omega_1}\det Dy(x)\,dx\right)\\&=&H(\det M).\ee

\section{Discussion}
\label{discussion}
 Theorem \ref{exthm} gives conditions under which a minimizer $y^*$ of 
 $I_{P_1,P_2}$ in $\mathcal A$ exists, but says nothing about the regularity properties of $y^*$. In the simpler problem of isotropic nonlinear elasticity essentially nothing is known. In particular it is an open question whether  minimizers  are smooth, or smooth outside some closed set of zero measure,  or even if the usual weak form of the Euler-Lagrange equation holds (though some forms of the Euler-Lagrange equation can be established \cite{p31}). The presence of the (possibly discontinuous) lower order terms due to the intensity maps makes the problem for $I_{P_1,P_2}$ even more challenging.
 
 Theorem \ref{generalM} shows that if the desirable property holds that for linearly related images $y^*$ is the corresponding linear map, then $\psi$ depends on $Dy$ only through $\det Dy$, that is only on local volume changes, so that in particular the hypothesis (H2) of Theorem \ref{exthm} does not hold. This suggests that a better model might be to minimize a functional such as 
 \be
 \label{second}
 E_{P_1,P_2}(y)=\int_{\om_1}\left(\psi(c_1(x),c_2(y(x)), \det Dy(x))+ |D^2y(x)|^2\right)\,dx,
 \ee
for which existence of a minimizer can be proved for low dimensions $n$, and for which minimizers of linearly related images could be proved under suitable hypotheses to be linear. This idea is explored in \cite{ballhornermaths}.

We remark that it is straightforward to prove variants of Theorem \ref{exthm} (a) for the case when $y$ is required to map a finite number of landmark points in $\om_1$ to corresponding points in $\om_2$ (see e.g. \cite{linetal2011}), and (b) for the case when $P_1$ is a template image that is to be compared to an unknown part of $P_2$ (such as in image registration). In the case (b), for example, one can minimize 
\be
\label{reg}
I(y)=\int_{\om_1}\psi(c_1(x), c_2(y(x)), Dy(x))\,dx
\ee
subject to the constraint that $y:\om_1\to\om$ is  a homeomorphism for some (unknown) subdomain $\Omega=a+\lambda R\Omega_1\subset\Omega_2$, where $a\in \R^n$, $R\in SO(n)$, $\alpha\leq\lambda\leq\beta$ and $0<\alpha<\beta$ are given. Here we consider the case when the unknown part of $P_2$ is to be compared to a rigid transformation, rotation and uniform magnification of the template, but one can equally handle the case of more general affine transformations, which may be appropriate for images viewed in perspective.  Such variants are also explored in \cite{ballhornermaths}.

Of course this work needs to be supplemented with appropriate numerical experiments on images. The numerical minimization of integrals such as \eqref{0} is not straightforward even without the presence of the (possibly discontinuous) intensity functions, and the fact that the minimization is to be carried out in the admissible set $\mathcal A$ of homeomorphisms, rather than, say, maps satisfying Dirichlet boundary conditions, presents additional difficulties. From a rigorous perspective, the numerical method should take into account the  possible occurrence of the Lavrentiev phenomenon, whereby the infimum of the energy in $\mathcal A$ might be strictly less than the infimum among Lipschitz maps in $\mathcal A$ (such as those generated by a finite-element scheme). For discussions see 
\cite{p29,baili2007,balcietal2022}. 

\subsubsection{Acknowledgements} We are grateful to Alexander  Belyaev, Jos\'e Iglesias, David Mumford, Ozan \"{O}ktem, Martin Rumpf, Carola Schonlieb and  Benedikt Wirth for their interest and helpful suggestions. CLH was supported by EPSRC through grant EP/L016508/1. 

 \bibliographystyle{splncs04}
 

\end{document}